\documentclass{amsart}
\usepackage{amsfonts}
\usepackage{blkarray}
\usepackage{float,hyperref}
\newtheorem{theorem}{Theorem}[section]
\newtheorem{lemma}[theorem]{Lemma}

\theoremstyle{definition}

\newtheorem{example}[theorem]{Example}

\theoremstyle{remark}
\newtheorem{remark}[theorem]{Remark}
\numberwithin{equation}{section}

\usepackage{url}
\begin{document}

\title[New binary self-dual codes of length $68$ via short Kharaghani array]{%
New extremal binary self-dual codes of length $68$ via short
Kharaghani array over $\mathbb{F}_{2}+u\mathbb{F}_{2}$}

\author{Abidin KAYA }
\address{Department of Computer Engineering, Bursa Orhangazi University,
16310, Bursa, Turkey } \email{abidin.kaya@bou.edu.tr}

\begin{abstract}
In this work, new construction methods for self-dual codes are
given. The methods use the short Kharaghani array and a variation of
it. These are applicable to any commutative Frobenius ring. We apply
the constructions
over the ring $\mathbb{F}_{2}+u\mathbb{F}_{2}$ and self-dual Type I $%
[64,32,12]_2$-codes with various weight enumerators obtained as Gray
images. By the use of an extension theorem for self-dual codes we
were able to construct $27$ new extremal binary self-dual codes of
length $68$. The existence of the extremal binary self-dual codes
with these weight enumerators was previously unknown.
\end{abstract}
\maketitle

\section{Introduction}

Self-dual codes constitute an interesting class of codes since the
appearance of \cite{conway} where an upper bound on the minimum
distance of a binary self-dual code is given. This type of codes are
related to various topics such as design theory, graph theory and
lattice theory. Recently, self-dual codes over rings have been used
to construct new codes. For some
of the works done in this direction we refer the reader to \cite%
{dougherty2,han,karadenizfour,kyp}.

The upper bound on the minimum distance of a binary self-dual code
is finalized in \cite{Rains}. A binary self-dual code meeting the
bound is called extremal. The possible weight enumerators of
self-dual codes of lengths up to $64$ and $72$ have been listed in
\cite{conway}. Since then researchers used different techniques to
construct self-dual codes. Huffman gave a survey on classification
of self-dual codes over various alphabets in \cite{huffman}.
Construction of new self-dual codes and the classification of
self-dual codes have been a dynamic research area. Among those
constructions the ones using circulant matrices are the most
celebrated. In\ \cite{crnkovic}, binary self-dual codes of length
$72$ are constructed by Hadamard designs. Using automorphism groups
is another way to build up self-dual codes. We refer to
\cite{Betsumiya,dougherty1,dougherty3,gurel,kyp} for more
information.

In this work, inspired by a four-block circulant construction in \cite%
{Betsumiya} that uses Goethals-Seidel array we propose a new
construction via short Kharaghani array. A variation of the method
is also given. By using the methods for the ring
$\mathbb{F}_{2}+u\mathbb{F}_{2}$ we construct self-dual codes of
length $32$. As binary images of the extensions of these codes we
were able to construct $27$ new extremal binary self-dual codes of
length $68$. Self-dual codes for these weight enumerators have been
obtained for the firs time in the literature.

The rest of the work is organized as follows. In Section 2, the
preliminaries about the structure of the ring $\mathbb{F}_{2}+u\mathbb{F}%
_{2} $ and the construction from \cite{Betsumiya} that we were
inspired by are given. Section 3 is devoted to the methods we
introduce which use the short Kharaghani array. The computer algebra
system MAGMA \cite{magma} have been used for computation and results
regarding the constructions are given
in Section 4. A substantial number of self-dual Type I $\left[ 64,32,12%
\right] _{2}$-codes and $27$ new extremal binary self-dual codes of length $%
68$ are constructed. Section 5 concludes the paper with some
possible lines of research.

\section{Preliminaries}

Throughout the text let $\mathcal{R}$ be a commutative Frobenius
ring. A linear code $\mathcal{C}$ of length $n$ over $\mathcal{R}$
is an $\mathcal{R} $-submodule of $\mathcal{R}^{n}$. Elements of
$\mathcal{C}$ are called codewords. Codes over $\mathbb{F}_{2}$ and
$\mathbb{F}_{3}$ are called binary and ternary, respectively.
Consider two arbitrary elements $x=\left( x_{1},x_{2},\ldots
,x_{n}\right) $ and $y=\left( y_{1},y_{2},\ldots ,y_{n}\right) $ of
$\mathcal{R}^{n}$. The Euclidean inner product is defined as
$\left\langle x,y\right\rangle _{E}=\sum x_{i}y_{i}$ and in this
paper the duality is understood in terms of the Euclidean inner
product. In other
words the dual of a code $\mathcal{C}$ of length $n$ is denoted as $\mathcal{%
C}^{\bot }$ and defined to be
\[
\mathcal{C}^{\bot }=\left\{ x\in \mathcal{R}^{n}\mid \left\langle
x,y\right\rangle _{E}=0\text{ for all }y\in \mathcal{C}\right\} .
\]%
A code $\mathcal{C}$ is said to be \emph{self-orthogonal} when $\mathcal{%
C\subset C}^{\bot }$ and \emph{self-dual} when $\mathcal{C=C}^{\bot
}$. An even self-dual code is said to be Type II if all the
codewords have weights divisible by $4$, otherwise it is said to be
Type I. For more information on
self-dual codes over commutative Frobenius rings we refer to \cite%
{dougherty3}.

The ring $\mathbb{F}_{2}+u\mathbb{F}_{2}$ is a characteristic $2$
ring of size $4$. The ring is defined as
$\mathbb{F}_{2}+u\mathbb{F}_{2}=\left\{ a+bu|a,b\in
\mathbb{F}_{2},u^{2}=0\right\} $ which is isomorphic to the quotient
$\mathbb{F}_{2}\left[ x\right] /\left( x^{2}\right) $. Type II
codes over $\mathbb{F}_{2}+u\mathbb{F}_{2}$ have been studied in \cite%
{dougherty2}. Some construction methods for self-dual codes over $\mathbb{F}%
_{2}+u\mathbb{F}_{2}$ are given in \cite{han}. Karadeniz et. al.
classified
self-dual four-circulant codes of length $32$ over $\mathbb{F}_{2}+u\mathbb{F%
}_{2}$ in \cite{karadenizfour}. For codes over $\mathbb{F}_{2}+u\mathbb{F}%
_{2}$ a duality preserving linear Gray map is given in
\cite{dougherty2} as follows:
\[
\varphi :\left( \mathbb{F}_{2}+u\mathbb{F}_{2}\right)
^{n}\rightarrow \mathbb{F}_{2}^{2n},\text{ }\varphi \left(
a+bu\right) =\left( b,a+b\right) \text{, \ }a,b\in
\mathbb{F}_{2}^{n}.
\]

In \cite{conway}, Conway and Sloane gave an upper bound on the
minimum Hamming distance of a binary self-dual code which was
finalized by Rains as follows:

\begin{theorem}
$($\cite{Rains}$)$ Let $d_{I}(n)$ and $d_{II}(n)$ be the minimum
distance of a Type I and Type II binary code of length $n$,
respectively. Then
\begin{equation*}
d_{II}(n)\leq 4\lfloor \frac{n}{24}\rfloor +4
\end{equation*}%
and
\begin{equation*}
d_{I}(n)\leq \left\{
\begin{array}{ll}
4\lfloor \frac{n}{24}\rfloor +4 & \text{if $n\not\equiv 22\pmod{24}$} \\
4\lfloor \frac{n}{24}\rfloor +6 & \text{if $n\equiv 22\pmod{24}$.}%
\end{array}%
\right.
\end{equation*}
\end{theorem}

Self-dual codes meeting these bounds are called \emph{extremal}.

For the rest of the work we let $R=\left( r_{ij}\right) $ be the
back
diagonal $\left( 0,1\right) $-matrix of order $n$ satisfying $r_{i,n-i+1}=1$%
, $r_{ij}=0$ if $j\neq n-i+1$. We are inspried by a construction for
self-dual codes given in \cite{Betsumiya} as follows:

\begin{theorem}
\cite{Betsumiya}\label{eightcirculant}Let $A,B,C,D$ be four $n$ by
$n$ circulant matrices satisfying
$AA^{T}+BB^{T}+CC^{T}+DD^{T}=-I_{n}$ then the
code generated by the matrix%
\begin{equation*}
G=\left(
\begin{array}{c}
I_{4n}\
\end{array}%
\begin{array}{|cccc}
A & BR & CR & DR \\
-BR & A & D^{T}R & -C^{T}R \\
-CR & -D^{T}R & A & B^{T}R \\
-DR & C^{T}R & -B^{T}R & A%
\end{array}%
\right)
\end{equation*}%
is a self-dual code.
\end{theorem}

$\lambda $-circulant matrices share the most of the properties of
circulant
matrices. For instance, they commute with each other for the same $\lambda $%
. Thus, the construction in Theorem \ref{eightcirculant} can easily
be extended to $\lambda $-circulant matrices. The construction uses
Goethals-Seidel array and we propose four-block-circulant
constructions in Section $3$.

\section{Self-dual codes via short Kharaghani array\label{constructions}}

In this section two constructions for self-dual codes over
commutative Frobenius rings are given.\ Kharaghani gave some arrays
for orthogonal designs in \cite{kharaghani}. The first construction
uses the short Kharaghani array and the second uses a variation of
the array. However the conditions of duality appear to be strict we
obtain good examples of self-dual codes over the ring
$\mathbb{F}_{2}+u\mathbb{F}_{2}$ and the binary field
$\mathbb{F}_{2}$. The given methods can be used for any commutative
Frobenius ring. Throughout the section let $\mathcal{R}$ denote a
commutative Frobenius ring.\ In the following, ternary self-dual
codes are given as examples in order to demonstrate that the methods
work for non-binary alphabets. A ternary self-dual $\left[
n,k,d\right] _{3}$-code is said to be \emph{extremal} if $d$ meets
the upper bound $d\leq 3\lfloor \frac{n}{12}\rfloor +3$.

We need the following Lemma from \cite{kyp};

\begin{lemma}
\cite{kyp}\label{lemma} Let $A$ and $C$ be $\lambda $-circulant
matrices then $C^{\prime }=CR$ is a $\lambda $-reverse-circulant
matrix~and it is
symmetric. Moreover, $AC^{\prime }-C^{\prime }A^{T}=0$. Equivalently, $%
ARC^{T}-CRA^{T}=0$.
\end{lemma}

\begin{theorem}
\label{short}(Construction I) Let $\mathcal{C}$ be the linear code over $%
\mathcal{R}$ of length $8n$ generated by the matrix in the following
form;
\begin{equation}
G:=\left( \enspace I_{4n}\enspace%
\begin{array}{|cccc}
A & B & CR & DR \\
-B & A & DR & -CR \\
-CR & -DR & A & B \\
-DR & CR & -B & A%
\end{array}%
\right)  \label{shortm}
\end{equation}%
where $A,B,C$ and $D$ are $\lambda $-circulant matrices over the ring $%
\mathcal{R}$ satisfying the conditions%
\begin{eqnarray*}
AA^{T}+BB^{T}+CC^{T}+DD^{T} &=&-I_{n}\text{ and} \\
AB^{T}-BA^{T}-CD^{T}+DC^{T} &=&0.
\end{eqnarray*}%
Then $\mathcal{C}$ is self-dual.
\end{theorem}

\begin{proof}
Let $M$ be the right half of the matrix $G$ in \ref{shortm} then it
is enough to show that $MM^{T}=-I_{4n}$.
\begin{eqnarray*}
MM^{T} &=&\left(
\begin{array}{cccc}
A & B & CR & DR \\
-B & A & DR & -CR \\
-CR & -DR & A & B \\
-DR & CR & -B & A%
\end{array}%
\right) \left(
\begin{array}{cccc}
A^{T} & -B^{T} & -RC^{T} & -RD^{T} \\
B^{T} & A^{T} & -RD^{T} & RC^{T} \\
RC^{T} & RD^{T} & A^{T} & -B^{T} \\
RD^{T} & -RC^{T} & B^{T} & A^{T}%
\end{array}%
\right) \\
&=&\left(
\begin{array}{cccc}
X & Y & Z & T \\
-Y & X & -T & -Z \\
-Z & -T & X & Y \\
-T & Z & -Y & X%
\end{array}%
\right) ,\text{
\begin{tabular}{l}
$X=AA^{T}+BB^{T}+CC^{T}+DD^{T}$ \\
$Y=-AB^{T}+BA^{T}+CD^{T}-DC^{T}$ \\
$Z=-ACR-BDR+CRA^{T}+DRB^{T}$ \\
$T=-ADR+BCR-CRB^{T}+DRA^{T}.$%
\end{tabular}%
}
\end{eqnarray*}%
$Z=T=0$ by Lemma \ref{lemma} and $Y=0$, $X=-I_{n}$ by the assumption. Hence $%
MM^{T}=-I_{4n}$ which implies $GG^{T}=0$. Therefore, the code
$\mathcal{C}$ is self-orthogonal and self-dual due to its size.
\end{proof}

In the following example we obtain an extremal ternary self-dual
code of length $56$ by Theorem \ref{short}.

\begin{example}
Let $\mathcal{C}_{56}$ be the code over $\mathbb{F}_{3}$ obtained by
Construction I for $n=7,\ \lambda =1,$ $r_{A}=\left( 2200120\right)
,\
r_{B}=\left( 0020102\right) ,\ r_{C}=\left( 0010020\right) $ and $%
r_{D}=\left( 2111001\right) $. Then $\mathcal{C}_{56}$ is a
self-dual $\left[ 56,28,15\right] _{3}$-code. In other words, it is
an extremal ternary self-dual code of length\ $56$ with $68544$
words of weight $15$ and an automorphism group of order $2^{3}\times
7$.
\end{example}

Now we give a variation of the construction in the Theorem
\ref{short} as follows:

\begin{theorem}
\label{variation}(Construction II) Let $\lambda $ be an element of the ring $%
\mathcal{R}$ with $\lambda ^{2}=1$ and $\mathcal{C}$ be the linear
code over $\mathcal{R}$ of length $8n$ generated by the matrix;
\begin{equation}
G:=\left( \enspace I_{4n}\enspace%
\begin{array}{|cccc}
A & B & CR & DR \\
-B^{T} & A^{T} & DR & -CR \\
-CR & -DR & A & B \\
-DR & CR & -B^{T} & A^{T}%
\end{array}%
\right)  \label{variationm}
\end{equation}%
where $A,B,C$ and $D$ are $\lambda $-circulant matrices over
$\mathcal{R}$ satisfying the conditions
\begin{eqnarray*}
AA^{T}+BB^{T}+CC^{T}+DD^{T} &=&-I_{n}\text{ } \\
CD^{T}-DC^{T} &=&0\text{ and} \\
-ADR+BCR-CRB+DRA &=&0\text{.}
\end{eqnarray*}%
Then the code $\mathcal{C}$ is a self-dual code over $\mathcal{R}$.
\end{theorem}

\begin{proof}
Let $M$ be the right half of the matrix $G$ in \ref{variationm} then
\begin{equation*}
MM^{T}=\left(
\begin{array}{cccc}
X & Y & Z & T \\
-Y^{T} & X & -T^{T} & -U \\
-Z & -T & X & Y \\
-T^{T} & U & -Y^{T} & X%
\end{array}%
\right) ,\text{%
\begin{tabular}{l}
$X=AA^{T}+BB^{T}+CC^{T}+DD^{T}$ \\
$Y=-AB+BA+CD^{T}-DC^{T}$ \\
$Z=-ACR-BDR+CRA^{T}+DRB^{T}$ \\
$T=-ADR+BCR-CRB+DRA$ \\
$U=B^{T}DR+A^{T}CR-DRB-CRA$%
\end{tabular}%
}
\end{equation*}%
By Lemma \ref{lemma} $Z=0$. The matrices $A^{T}$ and $B^{T}$ are
$\lambda ^{-1}$-circulant, they are $\lambda $-circulant since
$\lambda =\lambda
^{-1} $. Hence by Lemma \ref{lemma} $U=0$. $%
Y=-AB+BA+CD^{T}-DC^{T}=CD^{T}-DC^{T}$ since $\lambda $-circulant
matrices
commute. By the assumption $Y=0=T$ and $X=-I_{n}$. It follows that $%
MM^{T}=-I_{4n}$ which implies $\mathcal{C}$ is self-orthogonal. The code $%
\mathcal{C}$ is self-dual due to its size.
\end{proof}

There are only two extremal self-dual ternary codes of length $24$.
Those are the extended quadratic residue code and the Pless symmetry
code. In the following example we obtain both by Theorem
\ref{variation}.

\begin{example}
\label{ternary}Let $\mathcal{C}_{24}$ be the code over
$\mathbb{F}_{3}$ obtained by Construction II for $n=3,\ \lambda =2,$
$r_{A}=\left( 221\right) ,\ r_{B}=\left( 201\right) ,\ r_{C}=\left(
212\right) $ and $r_{D}=\left( 221\right) $. Let $\mathcal{D}_{24}$
be the code over $\mathbb{F}_{3}$ obtained by Construction II for
$n=3,\ \lambda =2,$ $r_{A}=\left( 200\right) ,\ r_{B}=\left(
112\right) ,\ r_{C}=\left( 102\right) $ and $r_{D}=\left(
110\right) $. Then $\mathcal{C}_{24}$ and $\mathcal{D}_{24}$ are self-dual $%
\left[ 24,12,9\right] _{3}$-codes. The code $\mathcal{C}_{24}$ is
Pless Symmetry code and the code $\mathcal{D}_{24}$ is the extended
quadratic residue code over $\mathbb{F}_{3}$ for $p=23$.
\end{example}

\begin{remark}
Two extremal self-dual $\left[ 24,12,9\right] _{3}$-codes in Example \ref%
{ternary} are also easily obtained by Theorem \ref{short}. On the
other hand, only the Pless symmetry code of parameters $\left[
24,12,9\right] _{3}$ could be constructed by Theorem
\ref{eightcirculant}. That exhibits the constructions proposed in
this section might be advantageous compared to Theorem
\ref{eightcirculant} even if the conditions are restrictive.
\end{remark}

\section{Computational results}

The constructions given in Section \ref{constructions} can be
applied to any commutative Frobenius ring. We focus on binary
self-dual codes obtained by the methods. The constructions applied
to the binary field $\mathbb{F}_{2}$
and the ring $\mathbb{F}_{2}+u\mathbb{F}_{2}$. The results are tabulated. $%
27 $ new extremal binary self-dual codes of length $68$ are obtained
as an application of Theorem \ref{short} and Theorem
\ref{variation}.

In \cite{conway} the possible weight enumerators for a self-dual Type I $%
\left[ 64,32,12\right] _{2}$-code were characterized as:
\begin{eqnarray*}
W_{64,1} &=&1+\left( 1312+16\beta \right) y^{12}+\left(
22016-64\beta
\right) y^{14}+\cdots ,14\leq \beta \leq 284, \\
W_{64,2} &=&1+\left( 1312+16\beta \right) y^{12}+\left(
23040-64\beta \right) y^{14}+\cdots ,0\leq \beta \leq 277.
\end{eqnarray*}%
Recently, codes with $\beta =$29, 59 and 74 in $W_{64,1}$
\cite{kyp}, a code with $\beta =$80 in $W_{64,2}$ were constructed
in \cite{karadenizfour}, together with these the existence of codes
is known for $\beta =$14, 18, 22,
25,\ 29, 32, 36, 39, 44, 46, 53,\ 59, 60, 64 and 74 in $W_{64,1}$ and for $%
\beta =$0, 1, 2, 4, 5,\ 6, 8, 9, 10, 12, 13,\ 14, 16,\ 17, 18, 20,
21,\ 22, 23, 24,\ 25,\ 28,\ 19,\ 30, 32,\ 33,\ 36, 37, 38, 40,\ 41,\
44, 48, 51,\ 52,\ 56, 58, 64, 72, 80,\ 88,\ 96, 104, 108,\ 112,\
114,\ 118,\ 120 and 184 in $W_{64,2}$.

\subsection{Computational results for the Construction I}

Results for the Construction I for $n=8$\ over $\mathbb{F}_{2}$ and for $n=4$%
\ over $\mathbb{F}_{2}+u\mathbb{F}_{2}$ are given. Self-dual Type I
$\left[ 64,32,12\right] _{2}$-codes are constructed and tabulated.

\begin{table}[htbp]
\caption{Construction I over $\mathbb{F}_{2}$ for $n=8$}
\label{tab:table1}
\begin{center}
\begin{tabular}{||c|c|c|c|c|c||c||}
\hline
$\mathcal{C}_{i}$ & $r_{A}$ & $r_{B}$ & $r_{C}$ & $r_{D}$ & $\left\vert Aut(%
\mathcal{C}_{i})\right\vert $ & $\beta $ in $W_{64,2}$ \\
\hline\hline
$\mathcal{C}_{1}$ & $\left( 10001101\right) $ & $\left( 00010000\right) $ & $%
\left( 01000110\right) $ & $\left( 01111010\right) $ & $2^{5}$ & $0$ \\
\hline
$\mathcal{C}_{2}$ & $\left( 10111001\right) $ & $\left( 01111101\right) $ & $%
\left( 01100001\right) $ & $\left( 01111111\right) $ & $2^{5}$ & $16$ \\
\hline
$\mathcal{C}_{3}$ & $\left( 10110011\right) $ & $\left( 01101001\right) $ & $%
\left( 11101101\right) $ & $\left( 01101111\right) $ & $2^{6}$ & $16$ \\
\hline
$\mathcal{C}_{4}$ & $\left( 00100011\right) $ & $\left( 11010010\right) $ & $%
\left( 11110011\right) $ & $\left( 01010011\right) $ & $2^{5}$ & $32$ \\
\hline
$\mathcal{C}_{5}$ & $\left( 11011000\right) $ & $\left( 00001110\right) $ & $%
\left( 11010100\right) $ & $\left( 11000000\right) $ & $2^{5}$ & $48$ \\
\hline
$\mathcal{C}_{6}$ & $\left( 11011000\right) $ & $\left( 11110001\right) $ & $%
\left( 01000111\right) $ & $\left( 01011100\right) $ & $2^{7}$ & $80$ \\
\hline
\end{tabular}%
\end{center}
\end{table}
For $n=8$ Theorem \ref{short} gives self-dual codes over the binary field $%
\mathbb{F}_{2}$ that are listed in Table \ref{tab:table1}.

\begin{table}[htbp]
\caption{Construction I over $\mathbb{F}_{2}+u\mathbb{F}_{2}$ for
$n=4$} \label{tab:table2}
\begin{center}
\begin{tabular}{||c|c|c|c|c|c|c||c||}
\hline
$\mathcal{D}_{i}$ & $\lambda $ & $r_{A}$ & $r_{B}$ & $r_{C}$ & $r_{D}$ & $%
\left\vert Aut(\mathcal{D}_{i})\right\vert $ & $\beta $ in $W_{64,2}$ \\
\hline\hline $\mathcal{D}_{1}$ & $3$ & $\left( 3,3,1,u\right) $ &
$\left( u,0,0,1\right) $
& $\left( 0,0,3,0\right) $ & $\left( 3,u,1,0\right) $ & $2^{5}$ & $0$ \\
\hline $\mathcal{D}_{2}$ & $3$ & $\left( 1,1,1,u\right) $ & $\left(
u,1,0,1\right) $
& $\left( u,3,3,0\right) $ & $\left( 0,u,1,1\right) $ & $2^{5}$ & $16$ \\
\hline $\mathcal{D}_{3}$ & $3$ & $\left( 3,1,3,u\right) $ & $\left(
0,1,u,1\right) $
& $\left( u,3,3,u\right) $ & $\left( u,u,1,1\right) $ & $2^{6}$ & $16$ \\
\hline $\mathcal{D}_{4}$ & $3$ & $\left( 3,1,3,0\right) $ & $\left(
0,1,u,1\right) $
& $\left( u,1,3,u\right) $ & $\left( u,0,1,1\right) $ & $2^{5}$ & $32$ \\
\hline $\mathcal{D}_{5}$ & $3$ & $\left( 1,3,1,0\right) $ & $\left(
u,1,0,1\right) $
& $\left( 0,3,3,u\right) $ & $\left( 0,0,3,1\right) $ & $2^{5}$ & $48$ \\
\hline $\mathcal{D}_{6}$ & $3$ & $\left( 3,1,3,u\right) $ & $\left(
u,3,0,3\right) $
& $\left( u,3,1,0\right) $ & $\left( u,0,1,3\right) $ & $2^{7}$ & $80$ \\
\hline
\end{tabular}%
\end{center}
\end{table}
In Table \ref{tab:table2} Construction II is applied to the ring $\mathbb{F}%
_{2}+u\mathbb{F}_{2}$ in order to construct self-dual codes of
length $32$.

\begin{remark}
The first extremal self-dual binary code of length $64$ with a
weight enumerator $\beta =80$ in $W_{64,2}$ is constructed in
\cite{karadenizfour} by using four circulant construction over
$\mathbb{F}_{2}+u\mathbb{F}_{2}$. In tables \ref{tab:table1} and
\ref{tab:table2} we give an alternative construction for the code by
the short Kharaghani array.
\end{remark}

\subsection{Computational results for the Construction II}

In this section we give the computational results for the
Construction II.

\begin{table}[h]
\caption{Construction II over $\mathbb{F}_{2}$ for $n=8$}
\label{tab:table3}
\begin{center}
\begin{tabular}{||c|c|c|c|c|c||c||}
\hline
$\mathcal{E}_{i}$ & $r_{A}$ & $r_{B}$ & $r_{C}$ & $r_{D}$ & $\left\vert Aut(%
\mathcal{E}_{i})\right\vert $ & $\beta $ in $W_{64,2}$ \\
\hline\hline
$\mathcal{E}_{1}$ & $\left( 00000010\right) $ & $\left( 01101100\right) $ & $%
\left( 01100111\right) $ & $\left( 10110000\right) $ & $2^{4}$ & $0$ \\
\hline
$\mathcal{E}_{2}$ & $\left( 11101100\right) $ & $\left( 10101110\right) $ & $%
\left( 10111110\right) $ & $\left( 01111010\right) $ & $2^{5}$ & $0$ \\
\hline
$\mathcal{E}_{3}$ & $\left( 01110110\right) $ & $\left( 10101000\right) $ & $%
\left( 11110010\right) $ & $\left( 11001001\right) $ & $2^{6}$ & $0$ \\
\hline
$\mathcal{E}_{4}$ & $\left( 10010011\right) $ & $\left( 01110101\right) $ & $%
\left( 01000110\right) $ & $\left( 10011110\right) $ & $2^{4}$ & $8$ \\
\hline
$\mathcal{E}_{5}$ & $\left( 01111000\right) $ & $\left( 01110101\right) $ & $%
\left( 10000001\right) $ & $\left( 00100100\right) $ & $2^{5}$ & $8$ \\
\hline
$\mathcal{E}_{6}$ & $\left( 00110100\right) $ & $\left( 01011010\right) $ & $%
\left( 00010011\right) $ & $\left( 01000011\right) $ & $2^{4}$ & $16$ \\
\hline
$\mathcal{E}_{7}$ & $\left( 00110001\right) $ & $\left( 01011010\right) $ & $%
\left( 01101011\right) $ & $\left( 11100011\right) $ & $2^{5}$ & $16$ \\
\hline
$\mathcal{E}_{8}$ & $\left( 01000110\right) $ & $\left( 11000000\right) $ & $%
\left( 10110100\right) $ & $\left( 10101001\right) $ & $2^{4}$ & $24$ \\
\hline
$\mathcal{E}_{9}$ & $\left( 10100011\right) $ & $\left( 11111101\right) $ & $%
\left( 11111001\right) $ & $\left( 01011010\right) $ & $2^{5}$ & $24$ \\
\hline $\mathcal{E}_{10}$ & $\left( 01000110\right) $ & $\left(
11001101\right) $ &
$\left( 10111110\right) $ & $\left( 00011100\right) $ & $2^{5}$ & $32$ \\
\hline $\mathcal{E}_{11}$ & $\left( 01100100\right) $ & $\left(
10100101\right) $ &
$\left( 10011111\right) $ & $\left( 10101100\right) $ & $2^{4}$ & $40$ \\
\hline $\mathcal{E}_{12}$ & $\left( 11110000\right) $ & $\left(
00010011\right) $ &
$\left( 11110001\right) $ & $\left( 10110101\right) $ & $2^{5}$ & $48$ \\
\hline
\end{tabular}%
\end{center}
\end{table}
In Table \ref{tab:table3} extremal self-dual Type I codes of length
$64$ are constructed.

\begin{table}[htbp]
\caption{Construction II over $\mathbb{F}_{2}+u\mathbb{F}_{2}$ for
$n=4$} \label{tab:table4}
\begin{center}
\begin{tabular}{||c|c|c|c|c|c|c||c||}
\hline
$\mathcal{F}_{i}$ & $\lambda $ & $r_{A}$ & $r_{B}$ & $r_{C}$ & $r_{D}$ & $%
\left\vert Aut(\mathcal{F}_{i})\right\vert $ & $\beta $ in $W_{64,2}$ \\
\hline\hline $\mathcal{F}_{1}$ & $3$ & $\left( 0,0,1,0\right) $ &
$\left( 3,0,3,u\right) $
& $\left( u,u,0,1\right) $ & $\left( 1,0,1,3\right) $ & $2^{4}$ & $0$ \\
\hline $\mathcal{F}_{2}$ & $3$ & $\left( 1,0,1,u\right) $ & $\left(
u,3,1,1\right) $
& $\left( 1,1,u,0\right) $ & $\left( 0,u,1,3\right) $ & $2^{4}$ & $8$ \\
\hline $\mathcal{F}_{3}$ & $3$ & $\left( 1,0,3,u\right) $ & $\left(
u,1,3,3\right) $
& $\left( 1,3,u,u\right) $ & $\left( 0,0,1,1\right) $ & $2^{5}$ & $8$ \\
\hline $\mathcal{F}_{4}$ & $1$ & $\left( 1,0,0,u\right) $ & $\left(
0,0,1,1\right) $
& $\left( 3,1,1,3\right) $ & $\left( 0,u,1,1\right) $ & $2^{4}$ & $16$ \\
\hline $\mathcal{F}_{5}$ & $3$ & $\left( 0,u,u,1\right) $ & $\left(
u,1,3,3\right) $
& $\left( 0,3,0,0\right) $ & $\left( 1,u,1,u\right) $ & $2^{5}$ & $16$ \\
\hline $\mathcal{F}_{6}$ & $3$ & $\left( u,0,1,u\right) $ & $\left(
1,1,3,1\right) $
& $\left( 3,3,1,u\right) $ & $\left( 3,1,1,u\right) $ & $2^{4}$ & $24$ \\
\hline $\mathcal{F}_{7}$ & $3$ & $\left( 3,u,1,0\right) $ & $\left(
u,1,1,3\right) $
& $\left( 1,3,0,0\right) $ & $\left( u,u,3,3\right) $ & $2^{5}$ & $24$ \\
\hline $\mathcal{F}_{8}$ & $1$ & $\left( 3,0,0,u\right) $ & $\left(
u,0,1,3\right) $
& $\left( 3,1,1,3\right) $ & $\left( u,u,1,3\right) $ & $2^{5}$ & $32$ \\
\hline $\mathcal{F}_{9}$ & $3$ & $\left( 0,0,1,u\right) $ & $\left(
1,1,1,1\right) $
& $\left( 1,3,1,u\right) $ & $\left( 3,3,3,u\right) $ & $2^{5}$ & $48$ \\
\hline
\end{tabular}%
\end{center}
\end{table}
Now we apply the construction in Theorem \ref{variation} to the ring $%
\mathbb{F}_{2}+u\mathbb{F}_{2}$ and give the results in Table \ref%
{tab:table4}.

Construction II has an advantage over Construction I. However, the
conditions are strict, Construction II \ allows us to narrow down
the search area. We may fix the matrices $C$ and $D$ which satisfy
$CD^{T}-DC^{T}=0$ and search for the circulant matrices $A$ and $B$
which satisy the remaining necessary consditions. We present that in
the following example:

\begin{example}
Let $n=4,\ \lambda =1+u,\ C$ and $D$ be $\lambda $-circulant
matrices with
first rows $r_{C}=\left( 1,1+u,u\right) $ and $r_{D}=\left( 0,0,1,1\right) $%
, respectively. Then $CD^{T}-DC^{T}=0$. So we may search for $\lambda $%
-circulant matrices $A$ and $B$ that satisfy $%
AA^{T}+BB^{T}+CC^{T}+DD^{T}=-I_{n}$ and $-ADR+BCR-CRB+DRA=0$. For
each pair
of such matrices a self-dual code of length $32$ over $\mathbb{F}_{2}+u%
\mathbb{F}_{2}$ will be obtained by Construction II. Let $A$ and $B$ be $%
\lambda $-circulant matrices with the following first rows%
\begin{equation*}
\begin{tabular}{|c|c|l|}
\hline $r_{A}$ & $r_{B}$ & $\beta $ in $W_{64,2}$ \\ \hline $\left(
1,0,1+u,u\right) $ & $\left( u,1,1+u,1+u\right) $ &
\multicolumn{1}{|c|}{$8$} \\ \hline $\left( 1+u,u,1,0\right) $ &
$\left( u,1,1,1+u\right) $ & \multicolumn{1}{|c|}{$24$} \\ \hline
\end{tabular}%
\end{equation*}%
then we obtain two extremal binary self-dual $\left[ 64,32,12\right] _{2}$%
-codes with automorphism groups of order $2^{5}$ as Gray images.
Note that this approach reduce the search field remarkably from
$4^{16}=4294967296$ to $4^{8}=65536$.
\end{example}

\begin{remark}
Although the constructions I and II have more strict conditions than
the construction in Theorem \ref{eightcirculant}, computational
results indicate
that they are superior over the method given in Theorem \ref{eightcirculant}%
. Since the only one Type I $\left[ 64,32,12\right] _{2}$-code with
weight enumerator $\beta =8$ in $W_{64,2}$ is obtained by applying
the construction
that uses Goethals-Seidel array to $\mathbb{F}_{2}$ and $\mathbb{F}_{2}+u%
\mathbb{F}_{2}$.
\end{remark}

\subsection{New extremal binary self-dual codes of length 68}

In \cite{dougherty1} the possible weight enumerators of a self-dual
$\left[ 68,34,12\right] _{2}$-code is characterized as follows:
\begin{eqnarray*}
W_{68,1} &=&1+\left( 442+4\beta \right) y^{12}+\left( 10864-8\beta
\right)
y^{14}+\cdots ,104\leq \beta \leq 1358\text{,} \\
W_{68,2} &=&1+\left( 442+4\beta \right) y^{12}+\left( 14960-8\beta
-256\gamma \right) y^{14}+\cdots
\end{eqnarray*}%
where $0\leq \gamma \leq 11$ and $14\gamma \leq \beta \leq
1870-32\gamma $. Recently, new codes in $W_{68,2}$ are obtained in
\cite{gurel,kyp} together with these, codes exist for $W_{68,2}$
when
\begin{eqnarray*}
\gamma  &=&0,\ \beta =\text{11,17,22,33,44},\dots ,\text{%
158,165,187,209,221,231,255,303} \\
\text{ or }\beta  &\in &\left\{ 2m|m=\text{17, 20, 88, 99, 102, 110,
119,
136, 165 or }80\leq m\leq 86\right\} ; \\
\gamma  &=&1,\ \beta =\text{49,57,59},\dots ,\text{160 or } \\
\beta  &\in &\left\{ 2m|m=\text{27, 28, 29, 95, 96 or }81\leq m\leq
90\right\} ; \\
\gamma  &=&2,\ \beta =\text{65,69,71,77,81,159,186 or }\beta \in
\left\{
2m|30\leq m\leq 68,\text{ }70\leq m\leq 91\right\} \text{ or} \\
\beta  &\in &\left\{ 2m+1|42\leq m\leq 69,\text{ }71\leq m\leq
77\right\} ;
\\
\gamma  &=&3,\ \beta =\text{%
101,103,105,107,115,117,119,121,123,125,127,129,131,133,} \\
&&\text{137,141,145,147,149,153,159,193 or } \\
\beta  &\in &\left\{ 2m\left\vert
\begin{array}{c}
m=\text{44,45,47,48,50,51,52,54,}\dots \text{,72,74,75,} \\
\text{77,}\dots \text{,84,86,87,88,89,90,91,92,94,95,97,98}%
\end{array}%
\right. \right\} ; \\
\gamma  &=&4\text{, }\beta \in \left\{ 2m\left\vert
\begin{array}{c}
m=\text{43,48,49,51,52,54,55,56,58,60,61,62,} \\
\text{64,65,67,}\dots \text{,71,75,}\dots \text{,78,80,87,97}%
\end{array}%
\right. \right\} ;\text{ } \\
\gamma  &=&6\text{ with }\beta \in \left\{ 2m|m=\text{69, 77, 78, 79, 81, 88}%
\right\} \text{.}
\end{eqnarray*}

In this section, we obtain the $27$ new codes with weight enumerators for $%
\gamma =0$ and $\beta =$174, 180, 182, 184, 186, 188, 190, 192, 194;
$\gamma =1$ and $\beta =$50, 52, 184, 186, 188; $\gamma =2$ and
$\beta =$184, 188,
190, 192, 194, 196, 198, 200, 206, 208; $\gamma =3$ and $\beta =$98, 106; $%
\gamma =4$ and $\beta =$196 in $W_{68,2}$.

\begin{theorem}
$($\cite{dougherty3}$)$ \label{extension}Let $\mathcal{C}$ be a
self-dual code over $\mathcal{R}$ of length $n$ and $G=(r_{i})$ be a
$k\times n$
generator matrix for $\mathcal{C}$, where $r_{i}$ is the $i$-th row of $G$, $%
1\leq i\leq k$. Let $c$ be a unit in $\mathcal{R}$ such that $c^{2}=1$ and $%
X $ be a vector in $\mathcal{R}^{n}$ with $\left\langle X,X\right\rangle =1$%
. Let $y_{i}=\left\langle r_{i},X\right\rangle $ for $1\leq i\leq
k$. Then
the following matrix%
\begin{equation*}
\left(
\begin{array}{cc|c}
1 & 0 & X \\ \hline
y_{1} & cy_{1} & r_{1} \\
\vdots & \vdots & \vdots \\
y_{k} & cy_{k} & r_{k}%
\end{array}%
\right) ,
\end{equation*}%
generates a self-dual code $\mathcal{C}^{\prime }$ over
$\mathcal{R}$ of length $n+2$.
\end{theorem}

\begin{table}[tbph]
\caption{New extremal binary self-dual codes of length 68 by Theorem
Theorem \protect\ref{extension}} \label{tab:table5}
\begin{center}
\begin{tabular}{||c|c|c|c|c||c||}
\hline
$\mathcal{C}_{68,i}$ & $\mathcal{C}$ & $c$ & $X$ & $\gamma $ & $\beta $ \\
\hline\hline $\mathcal{C}_{68,1}$ & $\mathcal{D}_{6}$ & $1$ &
$\left( 13u11uu3331uu10133u330u31u010031\right) $ & $0$ & $174$ \\
\hline $\mathcal{C}_{68,2}$ & $\mathcal{D}_{6}$ & $3$ & $\left(
103u303u0001333u3u03uu1u000u3313\right) $ & $0$ & $180$ \\ \hline
$\mathcal{C}_{68,3}$ & $\mathcal{D}_{6}$ & $3$ & $\left(
u1331u01333u31113101100310u1uu33\right) $ & $0$ & $182$ \\ \hline
$\mathcal{C}_{68,4}$ & $\mathcal{D}_{6}$ & $3$ & $\left(
001u3010uu1u00313101100310u1uu33\right) $ & $0$ & $184$ \\ \hline
$\mathcal{C}_{68,5}$ & $\mathcal{D}_{6}$ & $3$ & $\left(
1u303u1uu00311103uu1uu3000uu1313\right) $ & $0$ & $186$ \\ \hline
$\mathcal{C}_{68,6}$ & $\mathcal{D}_{6}$ & $3$ & $\left(
301u1u1u00u1311u3uu30u10uu0u1333\right) $ & $0$ & $188$ \\ \hline
$\mathcal{C}_{68,7}$ & $\mathcal{D}_{6}$ & $3$ & $\left(
3u13u333100u03011uu1333u1u110uu0\right) $ & $0$ & $190$ \\ \hline
$\mathcal{C}_{68,8}$ & $\mathcal{D}_{6}$ & $3$ & $\left(
310110310uu1u33011331u00u3300001\right) $ & $0$ & $192$ \\ \hline
$\mathcal{C}_{68,9}$ & $\mathcal{D}_{6}$ & $3$ & $\left(
101010100003111u3uu1u03u000u3331\right) $ & $0$ & $194$ \\ \hline
$\mathcal{C}_{68,10}$ & $\mathcal{F}_{1}$ & $3$ & $\left(
uu00333011u1330uu0u10u0u013u1100\right) $ & $1$ & $50$ \\ \hline
$\mathcal{C}_{68,11}$ & $\mathcal{F}_{1}$ & $3$ & $\left(
u0uu333013u113uu00u3uuu00330310u\right) $ & $1$ & $52$ \\ \hline
$\mathcal{C}_{68,12}$ & $\mathcal{D}_{6}$ & $1$ & $\left(
31013uu3133uu30311011uu33u03uu11\right) $ & $1$ & $184$ \\ \hline
$\mathcal{C}_{68,13}$ & $\mathcal{D}_{6}$ & $3$ & $\left(
330330u3113uu30311u13u013003uu11\right) $ & $1$ & $186$ \\ \hline
$\mathcal{C}_{68,14}$ & $\mathcal{D}_{6}$ & $1$ & $\left(
330130u3333uu3u3310330u33uu10011\right) $ & $1$ & $188$ \\ \hline
$\mathcal{C}_{68,15}$ & $\mathcal{D}_{6}$ & $1$ & $\left(
3u3u1u100uu1133u301u0u3113131uu0\right) $ & $2$ & $184$ \\ \hline
$\mathcal{C}_{68,16}$ & $\mathcal{D}_{6}$ & $3$ & $\left(
3u1u3u10uuu1133u301uuu3131113uuu\right) $ & $2$ & $188$ \\ \hline
$\mathcal{C}_{68,17}$ & $\mathcal{D}_{6}$ & $3$ & $\left(
011u01u3330u1001330310u13u010011\right) $ & $2$ & $190$ \\ \hline
$\mathcal{C}_{68,18}$ & $\mathcal{D}_{6}$ & $1$ & $\left(
1u1u10300u01311u1u100u11333130uu\right) $ & $2$ & $192$ \\ \hline
$\mathcal{C}_{68,19}$ & $\mathcal{D}_{6}$ & $3$ & $\left(
0310010111uu10u131u310u310010011\right) $ & $2$ & $194$ \\ \hline
$\mathcal{C}_{68,20}$ & $\mathcal{D}_{6}$ & $1$ & $\left(
10301010000333301010u01313111u00\right) $ & $2$ & $196$ \\ \hline
$\mathcal{C}_{68,21}$ & $\mathcal{D}_{6}$ & $3$ & $\left(
u310u1u3130u1uu113u130u11uu30u33\right) $ & $2$ & $198$ \\ \hline
$\mathcal{C}_{68,22}$ & $\mathcal{D}_{6}$ & $1$ & $\left(
u110u10331u0100111u3100310010u33\right) $ & $2$ & $200$ \\ \hline
$\mathcal{C}_{68,23}$ & $\mathcal{D}_{6}$ & $1$ & $\left(
0130u3u311uu1uu1310330013u030011\right) $ & $2$ & $206$ \\ \hline
$\mathcal{C}_{68,24}$ & $\mathcal{D}_{6}$ & $3$ & $\left(
301u1u1u00u1311u3010u01333133uu0\right) $ & $2$ & $208$ \\ \hline
$\mathcal{C}_{68,25}$ & $\mathcal{D}_{1}$ & $3$ & $\left(
u3330030u10uu313010001uu1030u0u3\right) $ & $3$ & $98$ \\ \hline
$\mathcal{C}_{68,26}$ & $\mathcal{D}_{1}$ & $3$ & $\left(
1030uu1130u31311101u13u03030uu30\right) $ & $3$ & $106$ \\ \hline
$\mathcal{C}_{68,27}$ & $\mathcal{D}_{6}$ & $3$ & $\left(
u310u30313u030u311u130u130u10033\right) $ & $4$ & $196$ \\ \hline
\end{tabular}%
\end{center}
\end{table}

In Table \ref{tab:table5} the codes are generated over $\mathbb{F}_{2}+u%
\mathbb{F}_{2}$ by the matrices of the following form;
\[
G^{\prime }=\left(
\begin{array}{cc|c}
1 & 0 & X \\ \hline
y_{1} & cy_{1} &  \\
\vdots & \vdots & G \\
y_{k} & cy_{k} &
\end{array}%
\right)
\]%
where $G$ is the generating matrix of the code $\mathcal{C}$ with
the
specified circulant matrices. Then $\mathcal{C}_{68,i}$ is the binary image $%
\varphi \left( G^{\prime }\right) $ of the extension.

\begin{theorem}
The existence of extremal self-dual binary codes is known for $492$
parameters in $W_{68,2}.$
\end{theorem}

\begin{remark}
The binary generator matrices of the\ codes in Table
\ref{tab:table5} are available online at \cite{web}. Those have
automorphism groups of order $2$.
\end{remark}

\section{Conclusion}

Most of the constructions for self-dual codes are used to reduce the
search field. In this paper we use the short Kharaghani array and
determine the necessary conditions for duality. The constructions
could be used over different alphabets such as $\mathbb{Z}_{4}$; the
integers modulo $4$. One may suggest such constructions by using
various arrays. By such methods we may attempt to construct codes as
the extremal binary self-dual Type II codes of length $72$ which is
a long standing open problem.

\end{document}